\documentclass[10pt,reqno]{amsart}
\usepackage[colorlinks=true,
pdfstartview=FitV, linkcolor=cyan, citecolor=magenta,
urlcolor=]{hyperref}

\usepackage{amsmath,amsfonts,latexsym,amssymb}
\usepackage{amssymb,amsfonts, amsmath}
\usepackage{mathrsfs}
\usepackage[latin1]{inputenc}
\usepackage[T1]{fontenc}
\usepackage{ae,aecompl}
\usepackage{braket}
\usepackage{comment}
\usepackage{amssymb,amsfonts, amsmath}
 \usepackage{mathtools} 
\usepackage[all,arc]{xy}
\usepackage{enumerate}
\usepackage{mathrsfs}
\usepackage{mathabx}
\usepackage{amssymb}
\usepackage{dsfont}
\usepackage[dvipsnames]{xcolor}
\usepackage{graphicx}
\usepackage{subfig}

\usepackage{verbatim} 
\usepackage{amssymb,amsfonts, amsmath}
\usepackage[all,arc]{xy}
\usepackage{enumerate}
\usepackage{mathrsfs}
\usepackage{hyperref}
\usepackage{xstring}
\usepackage{amsmath}
\usepackage[makeroom]{cancel}
\usepackage[bb = fourier,cal = euler,scr = rsfs]{mathalfa}	
\newtheorem{theorem}{Theorem}[section]
\newtheorem{lemma}[theorem]{Lemma}
\newtheorem{proposition}[theorem]{Proposition}
\newtheorem{corollary}[theorem]{Corollary}

\newtheorem{question}[theorem]{Question}
\newtheorem{Example}[theorem]{Example}

\newtheorem{fact}[theorem]{Fact}

\newtheorem{rmk}[theorem]{Remark}
\renewcommand{\leq}{\leqslant}
\renewcommand{\geq}{\geqslant}

\long\def\@savemarbox#1#2{\global\setbox#1\vtop{\hsize\marginparwidth 
  \@parboxrestore\tiny\raggedright #2}}
\usepackage{braket}
\usepackage{amsmath}
\usepackage[title]{appendix}

\makeatletter
\renewcommand*\env@matrix[1][\arraystretch]{%
\edef\arraystretch{#1}%
\hskip -\arraycolsep
\let\@ifnextchar\new@ifnextchar
\array{*\c@MaxMatrixCols c}}
\makeatother
\AtEndDocument{\bigskip{\footnotesize%
  \textsc{Institut des Hautes \'Etudes Scientifiques, 
Universit\'e Paris-Saclay, 35 route de Chartres, 91440 Bures-sur-Yvette, France} \par  
 \textit{E-mail address}: \texttt{douba@ihes.fr} 
\vspace{0.2cm}

\textsc{CNRS, Laboratoire Alexander Grothendieck, Institut des Hautes \'Etudes Scientifiques, 
Universit\'e Paris-Saclay, 35 route de Chartres, 91440 Bures-sur-Yvette, France} \par  
 \textit{E-mail address}: \texttt{tsouvkon@ihes.fr}

  \textsc{}}}

\date{\today}
\title{On regular subgroups of $\mathsf{SL}_3(\mathbb{R})$}
\date{\today}
\author{Sami Douba}
\author{Konstantinos Tsouvalas}

\begin{document}

\frenchspacing

\maketitle

\begin{abstract} Motivated by a question of M. Kapovich, we show that the $\mathbb{Z}^2$ subgroups of $\mathsf{SL}_3(\mathbb{R})$ that are {\em regular} in the language of Kapovich--Leeb--Porti, or {\em divergent} in the sense of Guichard--Wienhard, are precisely the lattices in minimal horospherical subgroups. This rules out any relative Anosov subgroups of $\mathsf{SL}_3(\mathbb{R})$ that are not in fact Gromov-hyperbolic. By work of Oh, it also follows that a Zariski-dense discrete subgroup $\Gamma$ of $\mathsf{SL}_3(\mathbb{R})$ contains a regular $\mathbb{Z}^2$ if and only if~$\Gamma$ is commensurable to a conjugate of $\mathsf{SL}_3(\mathbb{Z})$. In particular, a Zariski-dense regular subgroup of $\mathsf{SL}_3(\mathbb{R})$ contains no $\mathbb{Z}^2$ subgroups. \end{abstract}

\section{Introduction}

Our discussion is motivated by the following question of M. Kapovich, also considered by D. Long and A. Reid. 

\begin{question}\label{kapovich} \textup{(}\cite[Prob.~3.3]{kapovich2023list}\textup{)}
Is there a subgroup of $\mathsf{SL}_3(\mathbb{Z})$ isomorphic to $\mathbb{Z}^2 * \mathbb{Z}$?
\end{question}

We remark that $\mathsf{SO}_{3,1}(\mathbb{Z})$ contains copies of $\mathbb{Z}^2 * \mathbb{Z}$, and hence so does $\mathsf{SL}_n(\mathbb{Z})$ for each $n \geq 4$. While we do not resolve Question \ref{kapovich}, we establish the following.

\begin{theorem}\label{noregular}
There is no regular subgroup of $\mathsf{SL}_3(\mathbb{R})$ isomorphic to $\mathbb{Z}^2 * \mathbb{Z}$. 
\end{theorem}

{\em Regularity} (defined with respect to a parabolic subgroup) is a form of discreteness for subgroups of, or representations into, noncompact semisimple Lie groups that coincides with discreteness in the rank-one setting, but is strictly stronger in higher rank. These subgroups already appear in work of Benoist \cite{benoist1997proprietes}, and are the {\em divergent} subgroups of Guichard and Wienhard \cite{GW12}; see Section \ref{prelim} for the precise definition. 
For instance, the aforementioned copies of $\mathbb{Z}^2 * \mathbb{Z}$ in $\mathsf{SL}_n(\mathbb{Z})$ for $n\geq 4$ are regular, as are Anosov representations of Gromov-hyperbolic groups \cite{Labourie, GW12} and, more generally, relative Anosov representations of relatively hyperbolic groups \cite{KL,Zhu}. On the other hand, a lattice in a Cartan subgroup of $\mathsf{SL}_3(\mathbb{R})$ is not regular. Indeed, we show that regular $\mathbb{Z}^2$ subgroups of $\mathsf{SL}_3(\mathbb{R})$ are of a very particular form. Recall that a (resp., {\em minimal, maximal}) {\em horospherical} subgroup of $\mathsf{SL}_3(\mathbb{R})$ is by definition the unipotent radical of a (resp., maximal, minimal) proper parabolic subgroup of the latter.

\begin{theorem}\label{horo} A representation $\rho:\mathbb{Z}^2\rightarrow \mathsf{SL}_3(\mathbb{R})$ is regular if and only if $\rho(\mathbb{Z}^2)$ is a lattice in a minimal horospherical subgroup of $\mathsf{SL}_3(\mathbb{R})$.  \end{theorem}

It follows from Theorem \ref{horo} and results\footnote{In greater detail, suppose that $\Gamma < \mathsf{SL}_3(\mathbb{R})$ is discrete and Zariski-dense, and that some minimal horospherical subgroup $U$ of $\mathsf{SL}_3(\mathbb{R})$ is $\Gamma$-compact, where, following Oh \cite{oh99}, we say that a closed subgroup $H$ of $\mathsf{SL}_3(\mathbb{R})$ is {\em $\Gamma$-compact} if $H/(H \cap \Gamma)$ is compact. Then Oh exhibits in \cite[Prop.~4.1]{oh99} a ($\Gamma$-compact) maximal horospherical subgroup $V$ of $\mathsf{SL}_3(\mathbb{R})$ containing $U$ such that the other minimal horospherical subgroup $U'$ of $\mathsf{SL}_3(\mathbb{R})$ contained in $V$ is also $\Gamma$-compact. By Zariski-density of $\Gamma$, there is some $\gamma \in \Gamma$ such that the $\Gamma$-compact minimal horospherical subgroups $U$ and $\gamma U' \gamma^{-1}$ are opposite to one another \cite[Prop.~3.6]{oh99}. One now applies the main theorem of \cite{oh96} to conclude that $\Gamma$ is commensurable to a conjugate of $\mathsf{SL}_3(\mathbb{Z})$. See also Benoist's survey \cite[Prop.~4.1]{benoist2021survey}.} of Oh \cite{oh96, oh99}  that any Zariski-dense discrete subgroup of $\mathsf{SL}_3(\mathbb{R})$ containing a regular~$\mathbb{Z}^2$ is in fact commensurable to an $\mathsf{SL}_3(\mathbb{R})$-conjugate  of $\mathsf{SL}_3(\mathbb{Z})$. Theorem \ref{noregular} now follows since any discrete $\mathbb{Z}^2\ast\mathbb{Z}$ in $\mathsf{SL}_3(\mathbb{R})$ is necessarily Zariski-dense, while $\mathbb{Z}^2\ast\mathbb{Z}$ cannot be realized as a lattice in $\mathsf{SL}_3(\mathbb{R})$, for instance, because groups of the latter form enjoy Kazhdan's property (T)~\cite{Kazhdan} (see also Furstenberg~\cite{Fu67}). Moreover, as regularity is inherited by subgroups, and since $\mathsf{SL}_3(\mathbb{Z})$ contains a lattice in a Cartan subgroup of $\mathsf{SL}_3(\mathbb{R})$, we deduce the following.

\begin{corollary}\label{noz2}
A Zariski-dense regular subgroup of $\mathsf{SL}_3(\mathbb{R})$ contains no $\mathbb{Z}^2$ subgroups.
\end{corollary}

We remark that if $F$ is a lattice in a minimal horospherical subgroup of $\mathsf{SL}_3(\mathbb{R})$, then the limit set of $F$ in the Furstenberg boundary of $\mathsf{SL}_3(\mathbb{R})$ is the set of all projective flags of the form $(z, \ell)$, where either the point $z \in \mathbb{P}(\mathbb{R}^3)$ is fixed and $\ell \subset \mathbb{P}(\mathbb{R}^3)$ varies among all projective lines in $ \mathbb{P}(\mathbb{R}^3)$ containing $z$, or the projective line $\ell$ is fixed and $z$ varies among all points of $\ell$; for the precise notion of limit set used here, see Section \ref{prelim}.
Thus, another consequence of Theorem \ref{horo} is that a relative Anosov subgroup $\Gamma$ of $\mathsf{SL}_3(\mathbb{R})$ contains no $\mathbb{Z}^2$ subgroups, since the limit set of such $\Gamma$ in the Furstenberg boundary of $\mathsf{SL}_3(\mathbb{R})$ is antipodal in the language of Kapovich--Leeb--Porti; see \cite{KL, Zhu}. 

It is known that a group admitting a relative Anosov representation is relatively hyperbolic with respect to a family of virtually nilpotent subgroups; see \cite{KL,Zhu}. Since polycyclic groups that lack~$\mathbb{Z}^2$ subgroups are virtually cyclic, and since groups that are hyperbolic relative to virtually cyclic (more generally, hyperbolic) subgroups are themselves hyperbolic, we conclude the following from the previous paragraph.

\begin{corollary}\label{rel-hyp}
Relative Anosov subgroups of $\mathsf{SL}_3(\mathbb{R})$ are Gromov-hyperbolic.
\end{corollary}

In fact, in forthcoming work \cite{TZ} of the second-named author with F. Zhu, Corollary \ref{rel-hyp} is used to prove the stronger statement that a relative Anosov subgroup of $\mathsf{SL}_3(\mathbb{R})$ is virtually a free group or a surface group.

The relevance of Theorem \ref{horo} to Question \ref{kapovich} is further explained by the following proposition.


\begin{proposition}\label{combination}
Let $\Gamma$ be a lattice in a real linear algebraic semisimple Lie group $G$ of noncompact type and $P$ be a proper parabolic subgroup of $G$. Assume that $P$ is conjugate to its opposite. If $\Delta < \Gamma$ is $P$-regular in $G$ and there is a point in $G/P$ that is opposite to each point in the limit set of $\Delta$ in $G/P$, then for some $\gamma \in \Gamma$, the subgroup $\langle \Delta, \gamma \rangle < \Gamma$ decomposes as~$\Delta * \langle \gamma \rangle$. 
\end{proposition}

Thus, if there had been a regular $\mathbb{Z}^2$ in $\mathsf{SL}_3(\mathbb{Z})$ with ``small'' limit set in $\mathbb{P}(\mathbb{R}^3)$---a scenario that is ruled out by Theorem \ref{horo}---then Proposition \ref{combination} would have furnished a $\mathbb{Z}^2*\mathbb{Z}$ subgroup of~$\mathsf{SL}_3(\mathbb{Z})$, and even a regular such subgroup by work of Dey and Kapovich \cite[Thm.~3.2]{DK22}.

In light of Corollary \ref{noz2}, a result  \cite[Thm.~1.1]{ct20} of the second-named author with R. Canary asserting that Anosov subgroups of $\mathsf{SL}_3(\mathbb{R})$ are virtually isomorphic to Fuchsian groups, and aforementioned forthcoming work of the second-named author with Zhu, the following question seems natural.

\begin{question}
Is every regular Zariski-dense subgroup of $\mathsf{SL}_3(\mathbb{R})$ virtually isomorphic to a Fuchsian group?
\end{question}

\noindent {\bf Acknowledgements.} We thank Nic Brody, Alan Reid, Gabriele Viaggi, and Feng Zhu for interesting discussions. The first-named author was supported by the Huawei Young Talents Program. The second-named author was supported by the European Research Council (ERC) under the European's Union Horizon 2020 research and innovation programme (ERC starting grant DiGGeS, grant agreement No 715982).

\section{Preliminaries}\label{prelim}

For two sequences $(a_k)_{k\in \mathbb{N}}$ and $(b_k)_{k\in \mathbb{N}}$ of positive real numbers, we write $a_k \asymp b_k$ (resp., $a_k=O(b_k)$) if there is a constant $C>1$ such that $C^{-1}a_k \leq b_k\leq Ca_k$ (resp., $a_k\leq Cb_k$) for every $k$. 

Throughout this section, let $G$ be a finite-center real semisimple Lie group with finitely many connected components and maximal compact subgroup $K<G$, and let $X = G/K$ be the associated symmetric space. Let $P$ be a proper parabolic subgroup of $G$, so that $P$ is the stabilizer in $G$ of a point $z \in \partial_\infty X$, where $\partial_\infty X$ denotes the visual boundary of $X$. Pick a point $o \in X$, and let $\xi$ be the geodesic ray in $X$ emanating from $o$ in the class of $z$. Fix also a Weyl chamber $\overline{\mathfrak{a}}^+ \subset X$ for $G$ in $X$ with origin $o$ containing the ray $\xi$. A sequence $(g_n)_{n \in \mathbb{N}}$ in $G$ is {\em $P$-regular} if the vector-valued distances $d_{\overline{\mathfrak{a}}^+}(o, g_n o)$ diverge from each wall of $\overline{\mathfrak{a}}^+$ not containing $\xi$. This notion is independent of all the choices made after specifying the parabolic subgroup $P$. If $\Gamma$ is a discrete group, a representation $\rho:\Gamma \rightarrow G$ is called {\em $P$-regular} if for every sequence $(\gamma_n)_{n\in \mathbb{N}}$ in $\Gamma$ with $\gamma_n\rightarrow \infty$, the sequence $(\rho(\gamma_n))_{n \in \mathbb{N}}$ is $P$-regular. We remark that such a representation has finite kernel and discrete image, and is moreover $P^{\mathrm{opp}}$-regular, where $P^{\mathrm{opp}}$ denotes a parabolic subgroup opposite to $P$. A subgroup $\Delta$ of $G$ is called {\em P-regular} if the inclusion $\Delta \xhookrightarrow{} G$ is $P$-regular. Notice that if a subgroup $\Gamma$ of $G$ is $P$-regular then so are all subgroups of $\Gamma$. For more background, see Kapovich--Leeb--Porti \cite{KLP17}, as well as earlier work of Guichard--Wienhard \cite{GW12} where the notion of $P$-regularity appears instead as {\em $P$-divergence}.


\begin{Example} \normalfont{(The case  $G=\mathsf{SL}_3(\mathbb{K})$). Let $\mathbb{K}=\mathbb{R}$ or $\mathbb{C}$ and set $K_{\mathbb{R}}=\mathsf{SO}(3)$ and $K_{\mathbb{C}}=\mathsf{SU}(3)$. Any element $g\in \mathsf{SL}_3(\mathbb{K})$ can be written in the form $$g=k_g\textup{diag}\big(\sigma_1(g),\sigma_2(g),\sigma_3(g)\big)k_g' \ \ k_g,k_g'\in K_{\mathbb{K}},$$ where $\sigma_1(g)\geq \sigma_2(g)\geq \sigma_3(g)$ are uniquely determined and are called the {\em singular values} of $g$. The {\em Cartan projection}\footnote{This is the vector-valued distance $d_{\overline{\mathfrak{a}}^+}(o, g o)$ with respect to a particular choice of point $o \in X := \mathsf{SL}_3(\mathbb{K})/K_{\mathbb{K}}$ and Weyl chamber for $\mathsf{SL}_3(\mathbb{K})$ in $X$ with origin $o$.} of $g$ is $\mu(g)=(\log \sigma_1(g),\log \sigma_2(g),\log \sigma_3(g))\in \overline{\mathfrak{a}}^{+}$.} 

We will simply say a sequence in (resp., a representation into, subgroup of) $\mathsf{SL}_3(\mathbb{K})$ is {\em regular} if it is $P$-regular with respect to the stabilizer $P<\mathsf{SL}_3(\mathbb{K})$ of a line in $\mathbb{K}^3$. This language is unambiguous for representations into (hence subgroups of) $\mathsf{SL}_3(\mathbb{K})$; indeed, if $P$ and $Q$ are any two proper parabolic subgroups of $\mathsf{SL}_3(\mathbb{K})$, then a representation $\rho: \Gamma \rightarrow \mathsf{SL}_3(\mathbb{K})$ is $P$-regular if and only $\rho$ is $Q$-regular. A sequence $(g_n)_{n\in \mathbb{N}}$ in $\mathsf{SL}_3(\mathbb{K})$ is regular if and only if $$\lim_{n \rightarrow \infty}\frac{\sigma_1(\rho(g_n))}{\sigma_2(\rho(g_n))}=\infty.$$ 
Note that, in this case, the sequence $\left(\frac{1}{\sigma_1(\rho(g_n))} \rho(g_n)\right)_{n \in \mathbb{N}}$ subconverges to a rank-$1$ matrix.
\end{Example}

We will also use the following characterization of $P$-regularity in terms of the dynamics on the flag manifold $G/P$. A sequence $(g_n)_{n\in \mathbb{N}}$ is called {\em$P$-contracting} if there are points $z^+ \in G/P$ and $z^- \in G/P^{\mathrm{opp}}$ such that $g_{n}$ converges uniformly on compact subsets of $C(z^-)$ to the constant function $z^+$, where $C(z^-)$ denotes the set of all points in $G/P$ opposite to $z^-$. In this case, we write $ (g_n)_n^+ := z^+$.

\begin{fact}\label{regularcontracting} \textup{(}\cite[Prop.~4.15]{KLP17}\textup{)}. A sequence in $G$ that is $P$-contracting is also $P$-regular. A sequence in $G$ that is $P$-regular possesses a subsequence that is $P$-contracting.

In particular, a sequence $(g_n)_{n \in \mathbb{N}}$ in $G$ is $P$-regular if and only if every subsequence of $(g_n)_{n \in \mathbb{N}}$ possesses a $P$-contracting subsequence. 
\end{fact}

The  {\em limit set} of a subgroup $\Gamma<G$  in the flag manifold $G/P$, denoted by $\Lambda_\Gamma^P$, is by definition the set of $(\gamma_n)_n^+ \in G/P$ for all $P$-contracting sequences $(\gamma_n)_{n \in \mathbb{N}}$ in $\Gamma$. If $P$ is conjugate to $P^\mathrm{opp}$, two subsets $\Lambda_1, \Lambda_2 \subset G/P$ are {\em antipodal} if each element of $\Lambda_1$ is opposite to each element of $\Lambda_2$.

The proof of the following lemma uses the fact that, for a matrix $g=(g_{ij})_{i,j=1}^3$ in $\mathsf{SL}_3(\mathbb{C})$, one has $\frac{1}{\sqrt{3}}||g||_2\leq \sigma_1(g)\leq ||g||_2$, where $||g||_2:=(\sum_{i,j=1}^{3}|g_{ij}|^2)^{1/2}$ is the $\ell^2$-matrix norm of $g$.

\begin{lemma}\label{1div} Let $(g_n)_{n\in \mathbb{N}}$ be an infinite unbounded sequence of matrices in $\mathsf{SL}_3(\mathbb{C})$ with $$g_n=\begin{pmatrix}
1 & x_{n} & y_{n} \\ 
 0& 1 & z_{n} \\ 
0 &  0& 1
\end{pmatrix}.$$ Then $(g_n)_{n\in \mathbb{N}}$ is regular if and only if $$\lim_{n\rightarrow \infty} \frac{x_n^2+y_n^2+z_n^2}{|x_n|+ |z_n|+|x_nz_n-y_n|}=\infty.$$\end{lemma} 

\begin{proof} A straightforward calculation shows that for every $n\in \mathbb{N}$ we have that $$g_n^{-1}=\begin{pmatrix}
1 & -x_n & x_nz_n-y_n \\ 
 0& 1 & -z_n \\ 
0 &  0& 1
\end{pmatrix}.$$ Since $g_n \in \mathsf{SL}_3(\mathbb{C})$, we have $\sigma_1(g_n)\sigma_2(g_n)\sigma_3(g_n)=1$ and $\sigma_1(g_n^{-1})=\sigma_3(g_n)^{-1}$ for every $n$, and hence we obtain $$\frac{\sigma_1(g_n)}{\sigma_2(g_n)}=\frac{\sigma_1(g_n)^2\sigma_3(g_n)}{\sigma_1(g_n)\sigma_2(g_n)\sigma_3(g_n)}=\frac{\sigma_1(g_n)^2}{\sigma_1(g_n^{-1})}.$$ Now since \begin{align*} \sigma_1(g_n) \asymp |x_n|+|y_n|+|z_n|, \> \ \sigma_1(g_n^{-1}) \asymp |x_n|+|x_nz_n-y_n|+|z_n|, \end{align*}
the conclusion follows.\end{proof}

\section{Proof of Theorem \ref{horo}}

\begin{proof}[Proof of Theorem \ref{horo}]
Suppose that $\rho:\mathbb{Z}^2\rightarrow \mathsf{SL}_3(\mathbb{R})$ is a regular representation. We first prove that the image of $\rho$ is unipotent. Fix a $\mathbb{Z}$-basis $x,y \in \mathbb{Z}^2$ for $\mathbb{Z}^2$. \medskip

\noindent {\em Claim 1. The image $\rho(\mathbb{Z}^2)$ is a unipotent subgroup of $\mathsf{SL}_3(\mathbb{R})$}.

\begin{proof}[Proof of Claim 1]
Suppose otherwise. Assume first that all the eigenvalues of $\rho(x)$ are distinct. Then, up to conjugation within $\mathsf{SL}_3(\mathbb{C})$, the image of $\rho$ is a diagonal subgroup of $\mathsf{SL}_3(\mathbb{C})$. Since $\rho$ is discrete, we have that $\mu(\rho(\mathbb{Z}^2))$ contains the intersection of $\overline{\mathfrak{a}}^+$ with a lattice in $\mathfrak{a}$. It follows that~$\rho$ is not regular in this case.   

In the remaining case, up to conjugating $\rho$ within $\mathsf{SL}_3(\mathbb{R})$, we have
$$\rho(x)=\begin{pmatrix}
\lambda_x & 1 & 0 \\ 
 0& \lambda_x & 0 \\ 
0 &  0& \lambda_x^{-2} 
\end{pmatrix},\ \ \rho(y)=\begin{pmatrix}
\lambda_y & \alpha_{y} & 0 \\ 
 0& \lambda_y & 0 \\ 
0 &  0& \lambda_y^{-2}
\end{pmatrix},$$
for some $\lambda_x, \lambda_y, \alpha_y \in \mathbb{R}$. Then we have
$$
\rho(x^n y^m) = \lambda_x^n \lambda_y^m \begin{pmatrix}
1 & \lambda_x^{-1} n + \alpha_y \lambda_y^{-1} m & 0 \\ 
 0& 1 & 0 \\ 
0 &  0& \lambda_x^{-3n} \lambda_y^{-3m}
\end{pmatrix}$$
for $n, m \in \mathbb{Z}$. Now there is an infinite sequence of distinct pairs of integers $(n_k,m_k)_{k\in \mathbb{N}}$ such that $\lim_k (\lambda_x^{-1}n_k + \alpha_y \lambda_y^{-1}m_k) = 0$ and $\lim_k( \lambda_x^{n_k} \lambda_y^{m_k} ) = \infty$; note we can indeed ensure the latter, since otherwise discreteness of $\rho$ would be violated. Observe that $\sigma_1(\rho(x^{n_k}y^{m_k}))\asymp \lambda_{x}^{n_k}\lambda_y^{m_k}$ as $k\rightarrow \infty$ and that the sequence of matrices $\big(\frac{1}{\lambda_x^{n_k}\lambda_y^{m_k}}\rho(x^{n_k}y^{m_k})\big)_{k\in \mathbb{N}}$ converges to a matrix of rank $2$. In particular, the sequence $(\rho(x^{n_k}y^{m_k}))_{k\in \mathbb{N}}$ cannot be regular, so that $\rho$ is not regular.
\end{proof}

Therefore, the image of the representation $\rho:\mathbb{Z}^2 \rightarrow \mathsf{SL}_3(\mathbb{R})$ has to be unipotent. We show that $\rho(\mathbb{Z}^2)$ lies in a minimal horospherical subgroup of $\mathsf{SL}_3(\mathbb{R})$. Up to conjugation, we may assume that \begin{equation}\label{unip} \rho(x)=\begin{pmatrix}
1 & a_{x} & b_{x} \\ 
 0& 1 & c_{x} \\ 
0 &  0&1 
\end{pmatrix}, \ \rho(y)=\begin{pmatrix}
1 & a_{y} & b_{y} \\ 
 0& 1 & c_{y} \\ 
0 &  0&1 
\end{pmatrix},\end{equation} where $a_x,b_x,a_y,b_y\in \mathbb{R}$. Since $\rho(x)$ commutes with $\rho(y)$, we have that $a_{x}c_{y}=a_{y}c_{x}$. 
\medskip

\noindent {\em Claim 2. The identity $a_{y}c_{x}=a_xc_y=0$ holds}.

\begin{proof}[Proof of Claim 2] We prove the claim by contradiction. Assuming $a_yc_x\neq 0$, we will exhibit infinite sequences $(w_{m})_{m\in \mathbb{Z}}$ in $\mathbb{Z}^2$ such that $\big(\frac{\sigma_1}{\sigma_2}(\rho(w_m))\big)_{m\in \mathbb{Z}}$ has an infinite bounded subsequence.

Set $\lambda:=\frac{c_x}{a_x}=\frac{c_y}{a_y}\neq 0$. By conjugating the image of $\rho$ with the diagonal matrix $\textup{diag}(1,1,\lambda)\in \mathsf{GL}_3(\mathbb{R})$, we may assume that $a_x=c_x$ and $a_y=c_y$, and hence $$\rho(x)=\begin{pmatrix}
1 & a_{x} & b_{x} \\ 
 0& 1 & a_{x} \\ 
0 &  0&1 
\end{pmatrix},\  \rho(y)=\begin{pmatrix}
1 & a_{y} & b_{y} \\ 
 0& 1 & a_{y} \\ 
0 &  0&1 
\end{pmatrix}.$$ 

A straightforward calculation shows that, for $m,n \in \mathbb{Z}$,
$$\rho(x^{n})=\begin{pmatrix}
1 & n a_x & n\big(b_{x}-\frac{a_x^2}{2}\big)+\frac{n^2 a_x^2 }{2} \\ 
 0& 1 &  n a_x \\ 
0 &  0&1 
\end{pmatrix}, \  \rho(y^{m})=\begin{pmatrix}
1 & m a_y & n\big(b_{y}-\frac{a_y^2}{2}\big)+\frac{m^2 a_y^2 }{2} \\ 
 0& 1 &  m a_y \\ 
0 &  0&1 
\end{pmatrix},$$ $$\rho(x^{n}y^{m})=\begin{pmatrix}
1 & a(m,n) & b(m,n) \\ 
 0& 1 &  a(m,n) \\ 
0 &  0&1 
\end{pmatrix},$$
where \begin{align*} a(m,n)&:= n a_x+m a_y,\\ b(m,n)&:= mn a_xa_y+\frac{n^2 a_x^2}{2}+\frac{m^2 a_y^2}{2}+n\Big(b_{x}-\frac{a_x^2}{2}\Big)+m\Big(b_{y}-\frac{a_y^2}{2}\Big) \\ &= \frac{1}{2}\big(n a_x+ma_y\big)^2+n\Big(b_{x}-\frac{a_x^2}{2}\Big)+m\Big(b_{y}-\frac{a_y^2}{2}\Big) \\ &=\frac{1}{2}a(m,n)^2+\frac{B_{x}}{a_{x}}a(m,n)+m\Big(B_{y}-\frac{ a_y}{a_x}B_{x}\Big)\\ &=\frac{1}{2} \Big( \Big(a(m,n)+\frac{B_{x}}{a_x}\Big)^2 - \frac{2B_x^2}{a_x^2}+mZ_{x,y}\Big), \end{align*} where the constants $B_{x},B_{y},Z_{x,y}\in \mathbb{R}$ are defined as follows: \begin{align*} B_{x}&:=b_{x}-\frac{a_x^2}{2},\ B_{y}:=b_{y}-\frac{ a_y^2}{2}\\ Z_{x,y}&:=2\Big(B_{y}-\frac{ a_y}{a_x}B_{x}\Big).\end{align*}

Suppose first that $Z_{x,y}=0$, and choose infinite sequences $(k_m)_{m\in \mathbb{N}}, (r_m)_{m\in \mathbb{N}}$ of integers such that $$\big|a(k_m,r_m)\big| = \big|k_m a_x+r_m a_y\big| \leq 1$$ for every $m$. By our assumption that $Z_{x,y}=0$, we have that $(b(k_m,r_m))_{m\in \mathbb{Z}}$ is also bounded, and hence so is $(\rho(x^{k_m}y^{r_m}))_{m\in \mathbb{N}}$, violating our assumption that $\rho$ is discrete and faithful.

Now suppose that $Z_{x,y}\neq0$. Let $m\in \mathbb{Z}$ with $mZ_{x,y}<0$, and define $$n_{m}:=\left \lfloor -m\frac{a_y}{a_x}+\frac{1}{a_x}\sqrt{|mZ_{x,y}|} \right \rfloor$$
so that \begin{equation}\label{amnm}\Big|a(m,n_m)-\sqrt{|mZ_{x,y}|}\Big|=|a_x| \Big| n_m+\frac{a_y}{a_x}m-\frac{1}{a_x}\sqrt{|mZ_{x,y}|}\Big| \leq |a_x|.\end{equation}

Note that $|a(m,n_m)|\asymp \sqrt{|m|}$, and hence \begin{align*}
\big|b(m,n_m)\big|& \leq \frac{B_x^2}{a_x^2}+\frac{1}{2}\left|a(m,n_m)+\frac{B_{x}}{a_x}-\sqrt{|mZ_{x,y}|}\right|\cdot  \left|a(m,n_m)+\frac{B_{x}}{a_x}+\sqrt{|m Z_{x,y}|}\right|\\ & \leq  \frac{B_x^2}{a_x^2}+\left(|a_x|+\frac{|B_x|}{\big|a_x\big|}\right)\left( \big|a(m,n_m)\big|+\frac{\big|B_{x}\big|}{\big| a_x\big|}+\sqrt{|mZ_{x,y}|}\right)\\ &=O(\sqrt{|m|}), \ \  mZ_{x,y}\rightarrow -\infty,\end{align*}
where the second inequality follows from (\ref{amnm}).

Finally, we claim that the sequence $\rho(w_{m})_{m\in \mathbb{Z}}$, where $w_m:=x^{n_m}y^{m}$, has an infinite subsequence that is not regular. Indeed, for $m\in \mathbb{Z}$ with $mZ_{x,y}<0$, we have by Lemma \ref{1div} that $$\frac{\sigma_1(\rho(w_m))}{\sigma_2(\rho(w_m))}\asymp \frac{2a(m,n_m)^2+b(m,n_m)^2}{2\big|a(m,n_m)\big|+\big| a(m,n_m)^2 -b(m,n_m)\big|}$$ and the latter fraction remains bounded since $\big|a(m,n_m)\big|\asymp \sqrt{|m|}$ and $\big|b(m,n_m)\big|=O(\sqrt{|m|})$ as $mZ_{x,y}\rightarrow -\infty$.

We thus arrive at a contradiction, and so we conclude that $a_xc_y=a_yc_x=0$. \end{proof}

\noindent {\em Completing the proof of Theorem \ref{horo}.} We have reduced to the case that $\rho$ is as in (\ref{unip}) with $a_xc_y=a_yc_x=0$. 

Suppose first that $a_x=c_x=0$ and $a_yc_y\neq 0$. In this case, we may define a new representation $\rho':\mathbb{Z}^2\rightarrow \mathsf{SL}_3(\mathbb{R})$ given by $$\rho'(x)=\rho(xy), \  \rho'(y)=\rho(y).$$ Since $\rho$ is assumed to be regular, the same holds for $\rho'$. Now note that the $(1,2)$ and $(2,3)$ entries of $\rho'(x)$ and $\rho'(y)$ are non-zero, so that the representation $\rho'$ cannot be regular by Claim 1, a contradiction. By applying the same argument with $x$ and $y$ interchanged, we conclude that in fact $a_x=a_y=0$ or $c_x=c_y=0$ as desired.

Finally, we verify that if $\rho(\mathbb{Z}^2)$ is a lattice in a minimal horospherical subgroup of $\mathsf{SL}_3(\mathbb{R})$, then~$\rho$ is indeed regular. This follows immediately from Lemma \ref{1div}, but we present the following geometric argument that applies in any dimension. We first consider the case that $\rho(\mathbb{Z}^2)$ is a lattice in the unipotent radical of the stabilizer in $\mathsf{SL}_3(\mathbb{R})$ of a hyperplane in $\mathbb{R}^3$.
\medskip

\noindent {\em Claim 3. Let $U$ be the unipotent radical of the stabilizer in $\mathsf{SL}_d(\mathbb{R})$ of a hyperplane $V \subset \mathbb{R}^d$. Then any lattice $F$ in $U$ is $P$-regular, where $P$ is the stabilizer of a line in $\mathbb{R}^d$}.

\begin{proof}[Proof of Claim 3]
We identify the $U$-invariant affine chart $\mathbb{P}(\mathbb{R}^d)\smallsetminus \mathbb{P}(V)$ with $\mathbb{R}^{d-1}$, so that $U$ acts on $\mathbb{R}^{d-1}$ via translations. For a point $z \in \mathbb{R}^{d-1}$ and $R>0$, denote by $B(z, R)$ the Euclidean ball in~$\mathbb{R}^{d-1}$ of radius $R$ centered at $z$. Fix a point $z_0 \in \mathbb{R}^{d-1}$.

Now let $(\gamma_n)_{n \in \mathbb{N}}$ be a sequence in $F$ with $\gamma_n \rightarrow \infty$. Then, since $\mathbb{P}(\mathbb{R}^d)$ is compact, up to extraction, we have that $\gamma_n z_0 \rightarrow z^+$ for some $z^+ \in \mathbb{P}(\mathbb{R}^d)$. Moreover, since $F$ acts properly on $\mathbb{R}^{d-1}$, we in fact have $z^+ \in \mathbb{P}(V)$. 

We claim that $(\gamma_n)_{n \in \mathbb{N}}$ converges uniformly on compact subsets of $\mathbb{R}^d$ to the constant function~$z^+$. Indeed, let $W_n$ be a metric $\frac{1}{n}$-neighborhood of $z^+$ in $\mathbb{P}(\mathbb{R}^d)$ with respect to the Fubini--Study metric on the latter; viewed in our chosen affine chart, the boundary of $W_n$ is a two-sheeted hyperboloid for $n$ sufficiently large. It suffices to show that for any $n \in \mathbb{N}$, there is some $N \in \mathbb{N}$ such that $W_n \supset \gamma_N B(z_0,n) = B(\gamma_N z_0, n)$. But this is true since, given any $n \in \mathbb{N}$, there is some $m \in \mathbb{N}$ such that $B(z,n) \subset W_n$ for each $z \in W_m$.\end{proof}


In the remaining case, where $\rho(\mathbb{Z}^2)$ lies in the unipotent radical of the stabilizer of a {\em line} in $\mathbb{R}^3$, one argues as above with the dual representation $\rho^{\ast}$ instead of $\rho$, as $\sigma_{i}(\rho^{\ast}(\gamma))=\sigma_{4-i}(\rho(\gamma))^{-1}$ for $\gamma \in \mathbb{Z}^2$ and $1\leq i\leq 3$.
\end{proof}


\begin{rmk} \normalfont{Following the above approach, it is not difficult to see that if  $\langle a,b\rangle<\mathsf{SL}_3(\mathbb{R})$ is a discrete $\mathbb{Z}^2$ which is not contained in a minimal horospherical subgroup, then the limit set of $\langle a,b \rangle$ in $\mathbb{P}(\mathbb{R}^3)$ consists of at most three points.}\end{rmk}

\section{Proof of Proposition \ref{combination}}

To prove Proposition \ref{combination}, we use the following variant of the ping-pong lemma. Similar arguments appear in work of Dey and Kapovich \cite{DK22}, but we include them here for the convenience of the reader.

\begin{lemma}\label{pingpong}
Let $G$ be a Lie group acting continuously on a manifold $\mathcal{F}$. Suppose $\Gamma_1, \Gamma_2 < G$ are infinite\footnote{In fact, our argument requires only that $|\Gamma_i| > 2$ for $i=1,2$. The statement remains true if at least one of the~$\Gamma_i$ has size at least $3$.} and that there are closed nonempty disjoint subsets $C_1, C_2 \subset \mathcal{F}$ such that $\gamma_i C_j \subset C_i$ for $\gamma_i \in \Gamma_i \smallsetminus \{1\}$ and $i \neq j$. Then $\langle \Gamma_1, \Gamma_2 \rangle  < G$ is discrete and decomposes as $\Gamma_1 * \Gamma_2$. 
\end{lemma}

\begin{proof}
Let $\rho: \Gamma_1 * \Gamma_2 \rightarrow G$ be the map induced by the inclusions $\Gamma_i \subset G$ for $i=1,2$. Take a sequence $w_n \in \Gamma_1 * \Gamma_2$ and suppose for a contradiction that $w_n \neq 1$ for any $n \in \mathbb{N}$ but $\lim_n\rho(w_n) = 1 \in G$. Up to relabeling $\Gamma_1$ and $\Gamma_2$ and extracting a subsequence of $(w_n)_n$, we may assume that for some fixed $i \in \{1,2\}$ and each $n \in \mathbb{N}$, the first letter (read from the left) in the canonical form of $w_n$ belongs to $\Gamma_1 \smallsetminus \{1\}$ and the last belongs to $\Gamma_i \smallsetminus \{1\}$. 

Suppose first that $i=1$. Then $\rho(w_n)C_2 \subset C_1$ for each $n \in \mathbb{N}$. Selecting some $z \in C_2$, we then have $z = \lim_n \rho(w_n) z \in C_1$ since $C_1$ is closed, so that $z \in C_1 \cap C_2$, a contradiction.

Now suppose that $i=2$. Pick an element $\gamma_1 \in \Gamma_1 \smallsetminus \{1\}$, and let $w_n' = \gamma_1 w_n \gamma_1^{-1}$ for $n \in \mathbb{N}$. Note that we still have $\lim_n \rho(w_n') = 1$. If for some subsequence $(w'_{n_k})_k$ of $(w_n')_n$ the canonical form of~$w'_{n_k}$ has odd length for each $k \in \mathbb{N}$, then one obtains a contradiction as in the previous paragraph. Otherwise, there is some $N \in \mathbb{N}$ such that the first letter (read from the left) in the canonical form of~$w_n$ is $\gamma_1^{-1}$ for $n \geq N$. Now select $\gamma'_1 \in \Gamma_1 \smallsetminus \{1, \gamma_1\}$, and let $w''_n = \gamma_1' w_n (\gamma_1')^{-1}$ for $n \in \mathbb{N}$. Then again we have $\lim_n \rho(w_n'') = 1$, but now the canonical form of $w''_n$ has odd length for $n \geq N$, so that we arrive at a contradiction as in the previous paragraph. 
\end{proof}

\begin{proof}[Proof of Proposition \ref{combination}]
Since we have assumed that there is a point in $G/P$ opposite to each point in $\Lambda_\Delta^P$, we can find a compact neighborhood $W_0$ of $\Lambda_\Delta^P$ and a compact subset $U \subset G/P$ with nonempty interior such that $U$ and $W_0$ are antipodal; see \cite[Lem.~4.24]{DKL}. As in \cite[Rmk.~6.4]{DK22}, we have by $P$-regularity of $\Delta$ that $\delta U \subset W_0$ for each nontrivial element $\delta \in \Delta$ apart from a finite list $\delta_1, \ldots, \delta_k \in \Delta \smallsetminus \{1\}$.

For $i =1, \ldots, k$, let $Z_i$ be the set of all $z \in G/P$ such that $z$ is not opposite to $\delta_i z$. Since each of the $Z_i$ is a proper algebraic subset of $G/P$, we have that $U \smallsetminus \bigcup_{i=1}^k Z_i$ has nonempty interior. We can thus find a compact subset $V \subset U \smallsetminus \bigcup_{i=1}^k Z_i$ with nonempty interior such that $V$ and $\delta_i V$ are antipodal for $i=1, \ldots, k$. Setting $W = W_0 \cup \bigcup_{i=1}^k \delta_i V$, we then have that $V$ and $W$ remain antipodal in $G/P$. 

Now since $\Gamma$ is a lattice in $G$, there is an element $g \in \Gamma$ generating a $P$-regular cyclic subgroup with $\Lambda_{\langle g \rangle}^P \subset \mathring V$ (one can always choose {\em $P$-proximal} such $g \in \Gamma$, the existence of which already follows, for instance, from \cite[Lemma~1]{selberg60}). There is then some $N \in \mathbb{N}$ such that $g^n W \subset V$ for all $n \in \mathbb{Z}$ with $|n| \geq N$. Moreover, by design, we have $\delta V \subset W$ for each $\delta \in \Delta \smallsetminus \{1\}$. Setting $\gamma = g^N$, we conclude from Lemma \ref{pingpong} that $\langle \Delta, \gamma \rangle < \Gamma$ decomposes as $\Delta * \langle \gamma \rangle$. \end{proof}

\bibliography{regularbiblio}{}
\bibliographystyle{siam}

\end{document}